\documentclass[]{article}

\usepackage[english,german]{babel}
\usepackage[utf8]{inputenc}
\usepackage{amsmath, amsthm, amssymb, amsfonts}
\usepackage{thmtools}
\usepackage[colorlinks=true,linkcolor=blue,unicode]{hyperref}
\usepackage[capitalize]{cleveref}
\usepackage{nameref}
\usepackage{comment}
\usepackage{mathrsfs}
\usepackage{graphicx}
\usepackage{mathtools}
\usepackage{esint}
\usepackage{textcomp}
\usepackage{gensymb}
\usepackage{tikz}
\usepackage{wrapfig}
\usepackage{lipsum}
\usepackage{placeins}
\usepackage[babel=true]{csquotes}
\usetikzlibrary{decorations.pathreplacing}
\usepackage[backend=biber,style=alphabetic]{biblatex}
\usepackage{enumitem}

\setlist[enumerate, 1]{label = \roman*)}

\newcommand{\differential}{\mathrm{D}}

\DeclarePairedDelimiterXPP{\scp}[2]{{}}{\langle}{\rangle}{}{#1, #2}
\DeclarePairedDelimiterXPP{\norm}[1]{{}}{\lVert}{\rVert}{}{#1}
\DeclarePairedDelimiterXPP{\abs}[1]{{}}{\lvert}{\rvert}{}{#1}
\newcommand{\inftynorm}[1]{\norm*{#1}_{\infty}}

\newcommand{\lspace}{\mathrm{L}}
\newcommand{\bv}{\mathrm{BV}}
\newcommand{\testfunctions}{ \mathscr{D} }
\newcommand{\sobolevspace}{\mathrm{W}}
\newcommand{\distributions}{ \mathscr{D}' }
\newcommand{\cont}{\mathrm{C}}

\newcommand{\standardbasis}{e}

\newcommand{\hausdorffmeasure}{\mathcal{H}}

\newcommand{\boundary}{\partial}

\DeclarePairedDelimiterX{\setwithcondition}[2]{\lbrace}{\rbrace}{#1 : #2}

\newcommand{\closure}[1]{\overline{#1}}
\newcommand{\restr}[1]{ \vert_{ #1 } }

\newcommand{\projection}{\Pi}

\newcommand{\trace}[1]{ \mathrm{tr} \left( #1 \right) }

\newcommand{\mr}{\mathbb{R}}
\newcommand{\mn}{\mathbb{N}}
\newcommand{\mc}{\mathbb{C}}

\newcommand{\dd}{\,\mathrm{d}}

\newcommand{\sequence}[2]{ \left( #1_{#2}\right)_{ #2\in\mn } }

\newcommand{\converge}{\rightarrow}

\newcommand{\ma}{\mathbb{A}}
\newcommand{\Avariationmeasure}[1]{ \abs*{ \ma #1 } }
\newcommand{\Asymbolmap}[1]{ \ma \left[ #1 \right] }
\newcommand{\bva}{ \bv^{ \ma } }

\newcommand{\Anullspace}{ N( \ma ) }
\newcommand{\Atensor}[2]{ #1 \otimes_{ \ma } #2 }

\newtheorem{theorem}{Theorem}
\newtheorem{corollary}[theorem]{Corollary}

\newtheorem{proposition}[theorem]{Proposition}

\theoremstyle{definition}		

\newtheorem{definition}{Definition}
\newtheorem*{remark}{Remark}
\newtheorem{realremark}[theorem]{Remark}

\binoppenalty= \maxdimen
\relpenalty= \maxdimen

\title{Inequalities for \texorpdfstring{$ \bva $}{BVA}--Functions}
\author{Pascal Steinke}

\begin{document}
	
	\hypersetup{pageanchor=false}
	\selectlanguage{english}
	\maketitle
	
	\begin{abstract}
		In this paper, we prove trace-type Poincaré and Sobolev inequalities for the space of functions of bounded $ \ma $-variation. 
	\end{abstract}
	
	\section{Introduction}
	
	If $ \Omega  $ is an open, bounded and connected Lipschitz domain in $ \mr^{ d } $, then the Poincaré inequality states that for $ 1 \leq p  < \infty $, there exists some constant $ C(p, \Omega ) > 0 $  such that 
	\begin{equation*}
		\norm*{ 
			u 
		}_{ \lspace^{ p } ( \Omega ) }
		\leq
		C 
		\norm*{
			\nabla u
		}_{ \lspace^{ p } ( \Omega ) }
	\end{equation*} 
	for all $  u \in \sobolevspace^{ 1, p  } ( \Omega ) $ with $ \fint_{ \Omega } u \dd x = 0 $. 
	Without any additional boundary conditions, we need to ensure that the mean value integral of $ u $ is $ 0 $, otherwise any constant function would yield a quick counterexample. 
	
	Our main goal of this paper is to generalize this inequality in two ways: By replacing the gradient on the right-hand side with a more general first order differential operator $ \ma $ and finding suitable replacements for the mean value integral. What we will end up with is firstly the inequality
	
	\begin{equation*}
		\norm{ u - \projection_{ \Omega } u }_{ \lspace^{ 1 } ( \Omega ) }
		\leq
		\Avariationmeasure{ u } ( \Omega)
	\end{equation*}
	for functions $ u \in \bva ( \Omega ) $, where $ \projection_{ \Omega } $ denotes the orthogonal projection onto the nullspace of $ \ma $, see \Cref{uniform_poincare_inequality_for_bva} for the precise statement.
	
	In a recent paper by Breit, Diening and Gmeineder \cite{trace_for_bva}, the authors showed that under the algebraic condition of $ \mc $-ellipticity, there exists a trace operator on $ \bva $ on the so called NTA-domains, which include open and bounded Lipschitz domains. Combining this with the idea of Boulkhemair and Chakib in \cite{uniform_poincare} to only subtract the trace on a given Lipschitz hypersurface  $ \Gamma \subset \closure{ \Omega } $ in the Poincaré inequality, we arrive at 
	\begin{equation*}
		\norm*{
			u - \projection_{ \Gamma } \trace{ u }
		}_{ \lspace^{ 1 } ( \Omega ) }
		\leq
		C \Avariationmeasure{ u } ( \Omega),
	\end{equation*}
	on $ \bva ( \Omega ) $ for the case $ p = 1 $ and
	\begin{equation*}
		\norm{
			u  - \projection_{ \Gamma } \trace{ u }
		}_{ \lspace^{ p } ( \Omega ) }
		\leq
		C \norm{ \ma  u }_{ \lspace^{ p } ( \Omega ) }
	\end{equation*}
	on $ \sobolevspace^{ \ma, p } ( \Omega ) $ for $ 1 < p < \infty $,
	where  $ \mathrm{ tr } $ denotes the trace  of $  u $  on $  \Gamma$ and  $ \projection_{ \boundary \Omega } $ denotes the orthogonal projection onto the nullspace of $ \ma $ restricted to $ \boundary \Omega $, see \Cref{trace_poincare_inequality_for_bva} respectively \Cref{trace_poincare_inequality_p} for the precise statements.
	The proof will use the standard functional analysis proof via a contradiction and compactness argument.
	
	One direct consequence of this inequality is the coercivity of the functional $ \mathscr{ F } \colon \sobolevspace^{ \ma , 1 } (\Omega )\to \mr $ given by
	\begin{equation*}
		\mathscr{ F } [v]
		\coloneqq
		\int_{ \Omega }
		f( x , \ma v )
		\dd x,
	\end{equation*}
	where $ f $ satisfies the linear growth condition $ f ( x , p ) \geq C \abs{ p } $ and $ v $ has to satisfy suitable boundary conditions.
	
	Lastly we will prove the Sobolev inequality
	\begin{equation*}
		\norm{
			u
		}_{ \lspace^{ d / ( d - 1 ) } ( \Omega ) }
		\leq
		C \left( 
		\Avariationmeasure{ u } ( \Omega ) 
		+
		\norm*{
			\trace{ u }
		}_{ \lspace^{ 1 } ( \boundary \Omega ) }
		\right) 
	\end{equation*}
	for functions $ u \in \bva ( \Omega ) $ via an extension theorem for $ \bva $-functions.
	
	The structure of the paper is as follows: In \cref{bva_functions_section}, we will go into more detail about $ \ma $ and record some basic properties of the space $ \bva  $, followed up by a  closer  look at the algebraic properties of $ \ma $. In \cref{inequality_chapter}, we will first prove 2 versions of the Poincaré inequality for $ \bva $, one involving the values of the function on a set of positive measure and one involving the trace. Lastly we show how we can prove a Sobolev type inequality involving the trace.
	
	\section{\texorpdfstring{$ \bva $}{BVA}--Functions}
	\label{bva_functions_section}
	
	\subsection{Definition of \texorpdfstring{$ \bva $}{BVA}}
	\begin{definition}
		\label{def_of_A}
		We call a differential operator
		$ \ma $ 
		a \textbf{constant-coefficient, linear, homogeneous first order differential operator from 
			$ \mr^{N} $ 
			to 
			$ \mr^{k}$ } 
		if there exist fixed linear maps 
		$\ma_{ j } \colon \mr^{N} \to \mr^{k}$ with
		\begin{equation*}
			\ma = \sum_{j=1}^{d} \ma_{j} \partial_{j}.
		\end{equation*}
	\end{definition}
	
	From now on, $ \ma $ will always be used in the fashion we have just defined it.
	
	\begin{definition}
		\label{bva_def}
		Let 
		$ U \subseteq \mr^{ d } $
		be an open set. 
		We define the space of
		\textbf{functions of bounded $ \ma $-variation}
		as
		\begin{equation*}
			\bva ( U )
			\coloneqq
			\{
			u \in \lspace^{ 1} ( U, \mr^{ N } ) 
			\colon
			\ma u \in \mathcal{ M } ( U , \mr^{ k } ) 
			\},
		\end{equation*}
		where 
		$ \mathcal{ M } ( U , \mr^{ k } ) $ denotes the $ \mr^{ k } $-valued Radon measures of finite total variation on $ U $.
		For $ \ma^{ \ast } = \sum_{ j = 1 }^{ d } \ma_{ j }^{ \top } \partial_{j} $, the total variation of $ \ma u $ will be denoted by 
		\begin{equation*}
			\Avariationmeasure{ u } ( U )
			=
			\sup
			\setwithcondition*
			{ \int_{U} \scp{u} { \ma^{\ast} \varphi } \dd x }
			{ 	\varphi\in C_{c}^{1} ( U , \mr^{k} ),
				\ \inftynorm{ \varphi } \leq 1 }.
		\end{equation*}
		The norm we use on $ \bva ( U ) $ is given by
		\begin{equation*}
			\norm*{ u }_{ \bva ( U, \mr^{ N } ) }
			\coloneqq
			\norm*{ u }_{ \lspace^{ 1 } ( U , \mr^{ N } ) }
			+
			\Avariationmeasure{ u } ( U ).
		\end{equation*}
	\end{definition}
	
	Two important properties of $  \bva ( U ) $ are the well-known lower semicontinuity of the variation measure and the smooth approximation with respect to the strict convergence, which can be found in \cite[Theorem~2.8]{trace_for_bva}.
	
	\subsection{Algebraic Properties of the Symbol Map }
	
	In order to save ourselves work, we try to use results which we have already seen when working with the full gradient, i.e. the Sobolev spaces or BV-spaces. 
	However, one quickly realizes that this will not be easy since $ \ma $ can, for example, only represent parts of the gradient.
	This makes it necessary to introduce some additional conditions on 
	$ \ma $.
	The section follows the paper \cite{trace_for_bva}.
	
	\begin{definition}
		Let $ \mathbb{ K } \in \{ \mr , \mc \} $. 
		The differential operator $ \ma $ is called
		\textbf{$ \mathbb{ K } $-elliptic} 
		if for every 
		$ \xi \in \mathbb{ K }^{ d } \setminus \{ 0 \} $, 
		we have that the  symbol map
		\begin{equation*}
			\ma [ \xi ]
			=
			\sum_{ j = 1 }^{ d } 
			\xi_{ j } \ma_{ j } 
			\colon \mathbb{ K }^{ N } \to \mathbb{ K }^{ k }
		\end{equation*}
		is injective. We sometimes also write $\Atensor{ v }{ \xi } \coloneqq \Asymbolmap{ \xi } ( v ) $ for $ v \in \mathbb{K}^{ N } $.
		
		We say that $ \ma $ has 
		\textbf{finite-dimensional nullspace}
		if the kernel
		$ \Anullspace $
		is finite dimensional, i.e.
		\begin{equation*}
			\dim ( \Anullspace )
			\coloneqq
			\dim \left(
			\setwithcondition*
			{ v \in \distributions \left( \mr^{ d } ,  \mr^{ N } \right) }
			{ \ma v = 0 }
			\right)
			< 
			\infty.
		\end{equation*}
	\end{definition}
	
	The fashion in which this is used is by exploiting that all norms on a finite-dimensional real vector space are equivalent. 
	This will let us switch in between norms.
	
	The proofs of the following theorems can be found in 
	\cite{trace_for_bva}.
	
	\begin{theorem}
		\label{characterization_of_c_ellipticity}
		For our differential operator $ \ma $ from before, the following are equivalent:
		\begin{enumerate}
			\item 
			\label{cellip_one}
			$ \ma $ has finite-dimensional nullspace.
			
			\item
			\label{cellip_item_two}
			$ \ma $ is $ \mc $-elliptic.
			
			\item
			\label{cellip_item_three}
			There exists some $ l \in \mn $ with 
			$ \Anullspace \subseteq \mathscr{ P }_{ l } $,
			where $ \mathscr{ P }_{ l } $ 
			denotes the space of polynomials (in $ d $ variables) with degree less or equal to $ l $.
		\end{enumerate}
	\end{theorem}
	
	Two important examples for $ \mc $-elliptic operators are the gradient and the symmetric gradient $ 1/2 \left( \differential u + (\differential u)^{ \top } \right)$.
	
	\begin{theorem}
		\label{trace_for_bva}
		Assume that $ \ma $ is $ \mc $-elliptic. Let $ \Omega \subset \mr^{ d } $ be an open and bounded Lipschitz domain.
		Then there exists a trace operator
		\begin{equation*}
			\mathrm{ tr } 
			\colon
			\bva ( \Omega )
			\to
			\lspace^{ 1 } ( \boundary \Omega , \mr^{ N } )
		\end{equation*}
		which is continuous with respect to the strict convergence. Moreover it is the unique strictly continuous extension of the classical trace on $ \bva ( \Omega ) \cap \cont ( \closure{ \Omega } ) $.
	\end{theorem}
	
	\begin{corollary}[Extension by zero for $ \bva $]
		\label{extension_by_zero_for_bva}
		Assume that $ \ma $ is $ \mc $-elliptic. Let $ \Omega $ be an open and bounded Lipschitz domain.
		Then for a given $ u \in \bva ( \Omega ) $, its extension by zero onto $ \mr^{ d } $ is in $ \bva( \mr^{ d } ) $ with
		\begin{equation*}
			\ma u = \ma u \restr{ \Omega } + \Atensor{ \trace{ u } }{ \nu } \hausdorffmeasure^{ d- 1 } \restr{ \boundary \Omega }.
		\end{equation*}
		Here $ \nu $ denotes the unit outer normal of $ \Omega $.
	\end{corollary}
	
	We need a compact embedding, which has been proven in \cite[Theorem~1.1]{embeddings_for_a_weakly_differentiable_functions_on_domains}.
	
	\begin{theorem}
		\label{bva_compact_embedding}
		Assume that $ \ma $ is $ \mc $-elliptic.
		Let
		$ \Omega \subset \mr^{ d } $
		be an open and bounded Lipschitz domain. 
		Then the
		$ \lspace^{ 1 } $-closure
		of the unit ball in 
		$ \bva ( \Omega ) $
		is compact in 
		$ \lspace^{ 1 } ( \Omega ) $.
	\end{theorem}
	
	\section{Poincaré inequality}
	\label{inequality_chapter}
	
	\subsection{The Projection Operator on Cubes}
	\label{intro_of_the_projection}
	
	As mentioned in the introduction, we want to explore how to generalize the Poincaré inequality 
	\begin{equation*}
		\norm*{ 
			f - \fint_{ \Omega } f ( y ) \dd y 
		}_{ \lspace^{ 1 } ( \Omega ) }
		\leq
		C 
		\norm*{
			\nabla f
		}_{ \lspace^{ 1 } ( \Omega ) }
	\end{equation*} 
	to  $ \bva ( \Omega ) $.
	We quickly see that after replacing the gradient with $ \ma $, this inequality can not hold on
	$ \bva $
	if there exists some non-constant polynomial in $ \Anullspace $. 
	
	Every non-constant element of $ \Anullspace $ can give us a quick counterexample. This leads us to an idea on how to modify this inequality in order to introduce it for $ \bva $.
	If we pick an element of $ \Anullspace $, then we should always subtract it on the left hand side.
	Moreover, the inequality has to scale, and therefore what we subtract from our function has to linearly depend on $ f $.
	Of course on the right hand side, we should replace the gradient by $ \ma $.
	
	One idea would therefore be to swap out the mean-value integral with the projection onto the nullspace $ \Anullspace $.
	This is actually consistent with the classical Poincar\'e inequality since the kernel of the gradient is spanned orthonormally by the constant function
	$ \abs*{ \Omega }^{- 1/2 } $.
	Thus if we take the $ \lspace^{ 2 } $-projection of some square-integrable function $ f $ onto the nullspace of the gradient, we have
	\begin{equation*}
		\Pi ( f ) 
		=
		\scp{ f }{ 1 }_{ \lspace^{ 2 } ( \Omega ) } \frac{ 1 }{ \abs{ \Omega } }
		=
		\fint_{ \Omega } f \dd x.
	\end{equation*}
	
	Since we want to define the projection for funtions in $ \bva ( \Omega) $, an $ \lspace^{ 1 } $-estimate is desired.
	Let $ A $ be some measurable set and denote by $ \projection_{ A }u $ the  $ \lspace^{ 2  } (A, \mr^{  N}; \mu) $-projection
	(with $ \mu  $ being the $d$-dimensional Lebesgue or  $d-1$-dimensional Hausdorff measure and $ \mu ( A ) < \infty $) onto $ \Anullspace  $.  Then
	\begin{equation*}
		\int_{ A } 
		\abs*{ 
			\projection_{ A } u 
		}^{ 2 }
		\dd \mu 
		\leq
		\int_{ A } 
		\abs*{ 
			\projection_{ A } u 
		}^{ 2 }
		\dd \mu
		+
		\int_{ A } 
		\abs*{ 
			u - \projection_{ A } u 
		}^{ 2 }
		\dd \mu
		=
		\int_{ A }
		\abs{ u }^{ 2 }
		\dd \mu.	
	\end{equation*}
	
	Assume that $ \ma $ is $ \mc $-elliptic.  Then since $ \Anullspace $ is finite-dimensional by 
	\cref{characterization_of_c_ellipticity} 
	and because all norms on a finite-dimensional vector space are equivalent, there exists some constant 
	$ C > 0 $ 
	such that
	\begin{equation*}
		\norm*{ 
			\projection_{ A } u 
		}_{ \lspace^{ \infty } ( A ) }
		\leq
		C
		\int_{ A }
		\abs*{
			\projection_{ A } u
		}
		\dd \mu
	\end{equation*}
	holds for all $ u \in \lspace^{ 2 } ( A , \mr^{ N } ; \mu ) $.
	
	This now yields the $ \lspace^{ 1 } $-estimate since for some $ \lspace^{ 2 } $-orthonormal basis $ \standardbasis_{ 1 }, \dots, \standardbasis_{ l } $ of $ \Anullspace \restr{ A } $, we deduce by using the boundedness of $\standardbasis_{ 1 }, \dots, \standardbasis_{ l } $ that
	\begin{equation}
		\label{l1_estimate_projection}
		\int_{ A }
		\abs*{
			\projection_{ A } u
		}
		\dd \mu
		=
		\int_{ A }
		\abs*{
			\sum_{ j = 1 }^{ l }
			\scp{ u }{ e_{ j } }
			e_{ j }
		}
		\dd \mu
		\leq
		C
		\int_{ A }
		\abs{ u }
		\dd \mu.
	\end{equation}
	Here the constant depends only on the orthonormal basis that we have chosen for the functions in $ \Anullspace \restr{A} $.
	We can therefore extend $ \projection_{ A } $ onto
	$ \lspace^{ 1 } ( A , \mr^{ N } ; \mu ) $ 
	such that the above $ \lspace^{ 1 } $-estimate \eqref{l1_estimate_projection} still holds.

	\subsection{Generalized Poincar\'e Inequality}
	
	The main content of this section is the following theorem. Its proof generalizes the idea of Boulkhemair and Chakib in \cite{uniform_poincare}. This idea will also be used to prove \cref{trace_poincare_inequality_for_bva}.
	
	\begin{theorem}
		\label{uniform_poincare_inequality_for_bva}
		Let $ \Omega $ be an open, bounded and connected Lipschitz domain and let $ E \subseteq \Omega $ be a measurable subset with $ \abs*{ E } > 0 $.
		If $ \ma $ is $  \mc $-elliptic, then there exists some constant $ C > 0 $ such that for all $ u \in \bva ( \Omega ) $ we have
		\begin{equation*}
			\norm*{
				u -\projection_{ E } u  
			}_{ \lspace^{1} ( \Omega ) }
			\leq
			C
			\Avariationmeasure{ u } ( \Omega ).
		\end{equation*}
	\end{theorem}	
	
	\begin{remark}
		By an abuse of notation, we understand $ \projection_{ E } u $ as a function on $ \Omega $ even though it is a priori only defined on $ E $. This is justified since $ \Anullspace $ only consists of polynomials and these are uniquely determined by their values on a set of positive Lebesgue measure.
	\end{remark}
	\begin{proof}
		Assume that the result is false. 
		Then there exists a sequence
		\newline
		$ \sequence{ u }{ n } \subseteq \bva ( \Omega )$
		such that for all 
		$ n \in \mn$, 
		we have
		\begin{equation*}
			\norm*{
				u_{n} - \projection_{ E } u_{ n }
			}_{\lspace^{1} ( \Omega ) } 
			>
			n \Avariationmeasure{ u_{ n } } ( \Omega ).
		\end{equation*} 
		We define 
		\begin{equation*}
			w_{ n } 
			\coloneqq 
			\frac
			{u_{ n } - \projection_{ E } u_{ n }}
			{ \norm{ u_{ n } - \projection_{ E } u_{ n } }_{ \lspace^{ 1 } ( \Omega ) } }.
		\end{equation*}
		Since
		$ \ma \projection_{ E } u_{ n } = 0 $, 
		we obtain
		\begin{equation} 
			\label{inequality_1}
			\Avariationmeasure{ w_{ n } } ( \Omega )
			< 
			\frac
			{ 1 }
			{ n } 
			\quad \text{and} \quad
			\norm* { 
				w_{ n } 
			}_{ \lspace^{1} ( \Omega , \mr^{ N } ) } = 1 
		\end{equation}
		for all 
		$ n \in \mn $.
		Moreover 
		$ \sequence{ w }{ n } $
		is a bounded sequence in
		$ \bva ( \Omega) $,
		and therefore by 
		\cref{bva_compact_embedding},
		there exists some 
		$ w \in \lspace^{ 1 } ( \Omega ) $
		and a non-relabeled subsequence
		such that
		$ w_{ n } \converge w $
		in 
		$ \lspace^{ 1 } ( \Omega ) $.
		Since the projection is continuous with respect to $ \lspace^{ 1 } ( E ) $, we obtain
		\begin{equation} 
			\label{projection_of_w_zero}
			\projection_{ E } w
			= 
			\lim_{ n \to \infty } \projection_{ E } w_{ n } 
			=
			0
		\end{equation}
		in $ \lspace^{ 1 } ( E ) $.
		
		By the lower semicontinuity of the variation measure and inequality (\ref{inequality_1}) we obtain $ \Avariationmeasure{ w } = 0  $ and thus, since $ \Omega $ is connected, we may write $ w = f \restr{ \Omega } $ for some polynomial $ f \in \Anullspace $ (the proof of \Cref{characterization_of_c_ellipticity} translates to open sets).
		Thus we get by equation \ref{projection_of_w_zero} that
		\begin{equation*}
			w 
			=
			\projection_{ E }  w
			=
			0
		\end{equation*}
		almost everywhere on
		$ E $ for the polynomial $ w $.
		Since $ \abs*{ E } > 0 $ by assumption, we obtain that
		$ w $
		has to be zero already.
		
		But this is a contradiction to the $ \lspace^{ 1 } $-convergence, namely
		\begin{equation*}
			0
			=
			\int_{ \Omega } \abs{ w } \dd x 
			=
			\lim_{ n \to \infty } \int_{ \Omega } \abs{ w_{ n } }  \dd x
			=
			1,
		\end{equation*}
		which finishes our proof.
	\end{proof}

	\subsection{Poincar\'e Inequality: Trace Style}
	
	The inequality we have obtained is already nice, but we can get an even better result with the use of our trace operator from before. 
	It turns out that functions in $ \Anullspace $ are able to \enquote{see} Lipschitz hypersurfaces as we establish in \Cref{characterization_of_r_ellipticity} below.
	Therefore, it should suffice to catch  the projection on some Lipschitz hypersurface and subtract it as before.
	Our argument earlier was based on $ E $ having positive measure, which will not work here. Consequently we have to make use of the trace operator.
	
	\begin{proposition}
		\label{characterization_of_r_ellipticity}
		Let $\ma $ be as in \Cref{def_of_A} and assume that its kernel only consists of analytic functions. Then $ \ma $ is $ \mr $-elliptic if and only if for every non-empty Lipschitz hypersurface $ \Gamma $and every $ f \in \Anullspace \setminus \{ 0 \} $, the restriction of $ f $ to $ \Gamma $ is not zero.
	\end{proposition}
	
	\begin{proof}
		First assume that $ \ma $ is $ \mr$-elliptic. Let $ \Gamma $ be a non-empty Lipschitz hypersurface and let $ f \in \Anullspace $ such that $ f = 0 $ on $ \Gamma $. Take some $ x \in \Gamma $. Then due to the implicit function theorem, we have that $ \mathrm{rank} ( \differential f (x) ) \leq 1 $. Thus we  find $ \xi \in \mr^{ d } $ and $ v \in \mr^{ N } $ such that we may write $ \differential f ( x ) = v \otimes \xi $.
		But due to $ f \in \Anullspace $, we have
		\begin{equation*}
			0 = \ma f ( x )
			=
			\sum_{ j = 1 }^{ d } \ma_{ j } ( \partial_{ x_{ j } } f ( x ) )
			=
			\sum_{ j = 1 }^{ d }
			\xi_{ j } \ma_{ j } v,
		\end{equation*}
		and thus by the $ \mr $-ellipticity of $ \ma $ obain that either $ v = 0 $ or $ \xi = 0 $. Either way, we get $ D f ( x ) = 0 $. Since this argument holds for every $ x \in \Gamma $, we get that for all $ 1 \leq j \leq d $, $ \partial_{ x_{ j } } f = 0 $ on $ \Gamma $.
		Now we use that $ \partial_{ x_{ j } } f \in \Anullspace $, and thus an induction yields that all derivatives of $ f $ are zero on $ \Gamma $. 
		Since $ f $ is analytic and $ \Gamma $ contains at least one point, we obtain $ f=0 $.
		
		On the other hand suppose that $ \ma $ is not $ \mr $-elliptic. Then we find $ \xi \in \mr^{ d } \setminus \{ 0 \} $ and $ v \in \mr^{ N } \setminus \{ 0 \} $ such that
		$
		\sum_{ j = 1 }^{ d }
		\xi_{ j } \ma_{ j } v
		= 0 .
		$
		Define $ \Gamma \coloneqq \{ x \colon \scp{ \xi }{ x } = 0 \} $ and
		$
		f(x) \coloneqq \sum_{ j = 1 }^{ d } \xi_{ j } x_{ j } v.
		$
		Then $f $ is zero on $ \Gamma $ and $ f \in \Anullspace \setminus \{ 0 \} $, which finishes the proof.
	\end{proof}
	
	\begin{realremark}
		\label{choice_of_trace}
		Let $ \Omega \subset \mr^{ d } $ be an open and bounded Lipschitz domain. 
		Let $ \omega \subset \mr^{ d }$ be  an open set such that 
		$ \Omega \cap \omega \neq \emptyset $  has Lipschitz boundary and let $ \Gamma \subseteq \boundary \omega \cap \closure{ \Omega } $ be a non-empty Lipschitz hypersurface (see \Cref{figure_for_gamma}). Then there exists a trace $ \trace{ u } $ of any $ u \in \bva ( \Omega ) $ on $ \Gamma $ simply by taking the trace of $  u  $ with respect to $ \Omega \cap \omega $ and restricting it to $ \Gamma $. Note however that in general, there is no canonical choice of such a trace if only $ \Gamma $ is given.
		
		To see this, take for example $ \Omega = ( - 1 , 1 ) $ and $ \Gamma = \{ 0 \} $ and consider $ u = \chi_{ ( 0, 1 ) } \in \bv ( \Omega ) $. Then both $ 0 $ and $ 1 $ would be plausible traces on $ \Gamma $. Thus in this example we have to decide whether we take the trace from the left or right side.
	\end{realremark}
	
	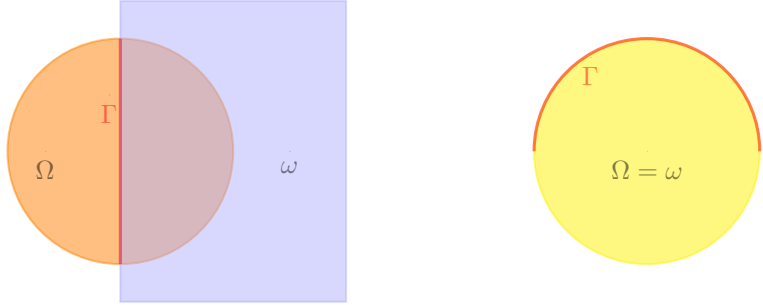
\begin{figure}
		\centering
		
		\begin{tikzpicture}[thick,  fill, opacity=0.5]
			\filldraw[orange] (0,0) circle (1.5cm);
			\filldraw[blue!30!white] (0,-2) rectangle (3,2);
			\draw[red, very thick] (0,-1.5) -- (0,1.5);
			
			\filldraw[black] (-1,0) circle (0pt)
			node[anchor=north, black] { $ \Omega $ };
			\filldraw[black] (2.25,0) circle (0pt)
			node[anchor=north, black] { $ \omega $ };
			\filldraw[black] (-0.15,0.75) circle (0pt)
			node[anchor=north, red] { $ \Gamma $ };
			
			\filldraw[yellow](7,0) circle (1.5cm);
			\draw[red, very thick] (8.5,0) arc (0:180:1.5cm);
			
			\filldraw[black] (7,0) circle (0pt)
			node[anchor=north, black] { $ \Omega = \omega $ };
			\filldraw[black] (6.25,1.25) circle (0pt)
			node[anchor=north, red] { $ \Gamma $ };
			
		\end{tikzpicture}
		\caption{The intersection of $ \boundary \omega  $  with $\closure{ \Omega } $  describes a Lipschitz hypersurface, of which we take a subset $ \Gamma $. The trace on $ \Gamma $ is taken with respect to $ \omega $.}
		\label{figure_for_gamma}
	\end{figure}
	
	After this remark, we can now state and prove our main result.
	
	\begin{theorem}
		\label{trace_poincare_inequality_for_bva}
		Let $ \Omega \subset \mr^{ d } $ be an open, bounded and connected Lipschitz domain. 
		Let $ \omega \subset \mr^{ d }$ be  an open set such that 
		$ \Omega \cap \omega \neq \emptyset $ has Lipschitz boundary and let $ \Gamma \subseteq \boundary \omega \cap \closure{ \Omega } $ be a non-empty Lipschitz hypersurface.
		Let $ \ma  $ be $ \mc $-elliptic and denote by $ \trace{ u } $ the trace of $ u $ on $ \Gamma $ as in \Cref{choice_of_trace}. Then there exists some constant $ C > 0 $ such that
		\begin{equation}
			\label{poincare_inequality_trace}
			\norm*{
				u - \projection_{ \Gamma } \trace{ u }
			}_{ \lspace^{ 1 } ( \Omega ) }
			\leq
			C \Avariationmeasure{ u } ( \Omega) 
		\end{equation}
		holds for all
		$ u \in \bva ( \Omega ) $.
	\end{theorem}
	
	\begin{remark}
		By an abuse of notation, we understand $ \projection_{ \Gamma } \trace{u} $ as a function on $ \Omega $ even though it is a priori only defined on $ \Gamma $. This is justified by \cref{characterization_of_r_ellipticity}.
	\end{remark}
	
	\begin{proof}
		Let us assume that such a constant does not exist.
		Then there exists
		$ \sequence{ u }{ n } \in \bva ( \Omega ) $
		such that
		\begin{equation*}
			\norm*{
				u_{ n } - \projection_{ \Gamma } \trace{ u_{ n } }
			}_{ \lspace^{ 1 }  ( \Omega )  }
			>
			n
			\Avariationmeasure{ u_{ n } } ( \Omega ) .
		\end{equation*}
		Defining 
		\begin{equation*}
			w_{ n } 
			\coloneqq 
			\frac
			{ u_{ n } - \projection_{ \Gamma } \trace{ u_{ n } } }
			{ \norm*{ u_{ n } - \projection_{ \Gamma } \trace{ u_{ n } } }_{ \lspace^{ 1 } ( \Omega ) } }
		\end{equation*}
		implies
		\begin{equation*}
			\norm*{ w_{ n } }_{ \lspace^{ 1 } ( \Omega ) }
			= 
			1
			\quad
			\text{ and }
			\quad
			\Avariationmeasure{ w_{ n } } ( \Omega ) 
			<
			\frac{ 1 }{ n }.
		\end{equation*}
		By compactness of $ \bva ( \Omega ) $ (\cref{bva_compact_embedding}), 
		we therefore know that there exists some
		$ w \in \bva ( \Omega ) $
		and some non-relabeled subsequence such that
		$ w_{ n } \to w $ in 
		$ \lspace^{ 1 } ( \Omega ) $.
		By the lower semicontinuity of the variation measure and the connectedness of $ \Omega $, we have
		$ w = f\restr{ \Omega } $ for some $ f \in \Anullspace $.
		
		Moreover, we know by
		\cref{trace_for_bva} 
		that the trace is continuous with respect to the strict convergence.
		Since we already have 
		\begin{equation*}
			\Avariationmeasure{ u } ( \Omega ) 
			=
			0
			=
			\lim_{ n \to \infty }
			\Avariationmeasure{ u_{ n } } ( \Omega )
		\end{equation*}
		and $ \lspace^{ 1 } ( \Omega ) $-convergence,
		this gives together with the continuity of the projection that
		\begin{align*}
			\norm*{ 
				\projection_{ \Gamma } \trace{ w } 
			}_{ \lspace^{ 1 } ( \Gamma ) }
			& =
			\liminf_{ n \to \infty }
			\norm*{
				\projection_{ \Gamma } \left(
				\trace{ w_{ n } } - \trace{ w }
				\right)
			}_{ \lspace^{ 1 } ( \Gamma ) }
			\\
			& \leq C
			\liminf_{ n \to \infty }
			\norm*{
				\trace{ w_{ n } } - \trace{ w }
			}_{ \lspace^{ 1 } ( \Gamma ) }
			\\
			& =
			0.
		\end{align*}
		In the first equality we used that $ \projection_{ \Gamma } ( \trace{ w_{ n } } ) = 0 $.
		We know that $ w \in \Anullspace $ and therefore 
		$ 0 = \projection_{ \Gamma } \trace{ w }= \trace{ w } $.
		This gives that $ w = 0 $ $\hausdorffmeasure^{ d - 1 } $-almost everywhere on $ \Gamma $.
		But by the injectivity of the restriction (\cref{characterization_of_r_ellipticity}) we then already have $ w = 0 $ in $ \Omega $.
		This leads to a contradiction as in \cref{uniform_poincare_inequality_for_bva}.
	\end{proof}	
	
	\begin{remark}
		In the case where $ \Omega $ is a ball and $ \Gamma = \boundary \Omega $, we get via a scaling argument that the constant $ C $  depends linearly on the radius of the ball.	
	\end{remark}
	
	Looking at the proof of the previous \Cref{trace_poincare_inequality_for_bva}, we moreover obtain the following generalization for the space $ \sobolevspace^{ \ma , p } ( \Omega ) \coloneqq \{ f \in \lspace^{ p }(\Omega) \colon \ma f \in \lspace^{ p } (\Omega ) \}$.
	
	\begin{corollary}
		\label{trace_poincare_inequality_p}
		Let $ 1  < p  < \infty $ and let $ \Omega \subset \mr^{ d } $ be an open, bounded and connected Lipschitz domain. Let $ \Gamma $ be some non-empty Lipschitz hypersurface with $ \Gamma \subseteq \closure{ \Omega } $. If $ \ma $ is $ \mc $-elliptic, then there exists some constant $ C > 0 $ such that
		\begin{equation*}
			\norm{
				u  - \projection_{ \Gamma } \trace{ u }
			}_{ \lspace^{ p } ( \Omega ) }
			\leq
			C \norm{ \ma  u }_{ \lspace^{ p } ( \Omega ) }
		\end{equation*}
		holds for all $ u \in \sobolevspace^{ \ma , p } ( \Omega ) $.
	\end{corollary}
	
	\begin{remark}
		Here the trace operator on $ \Gamma $ is unique since $ \sobolevspace^{ \ma, p } $-functions are not allowed to have jumps on $ d- 1 $-dimensional sets.
	\end{remark}
	
	\begin{proof}
		That the $ \lspace^{ 2 }  $-orthogonal projection onto the Nullspace of $ \ma $ can be extended to a bounded linear operator from $ \lspace^{ p } ( \boundary \Omega ) $  to $ \lspace^{ p }  ( \boundary \Omega ) $ follows with the same proof as before. The existence of a continuous trace operator and the compact embedding follow from \cite[Theorem~4.4]{sharp_trace_and_korn_inequalities_for_differential_operators}, where it has been proven that on smooth domains (the proof also works for sets with Lipschitz boundary), the space $ \sobolevspace^{ \ma , p } ( \Omega ) $ is $ \sobolevspace^{ 1, p } ( \Omega ) $ with equivalent norms for $ 1 < p < \infty $.
	\end{proof}
	
	\subsection{Sobolev Inequality: Trace Style}
	
	From the foregoing subsection, we can deduce an estimate of the form
	\begin{equation*}
		\norm*{
			u 
		}_{ \lspace^{ 1 } ( \Omega ) }
		\leq C
		\left(
		\Avariationmeasure{ u } ( \Omega ) 
		+
		\norm*{
			\trace{ u }
		}_{ \lspace^{ 1 } ( \boundary \Omega ) }
		\right).
	\end{equation*}
	This estimate however does not scale,  thus we want to explore whether or not it is possible to obtain an estimate in some better $ \lspace^{ p } $-space on the left-hand side.
	For this we first cite a suitable version of the Sobolev inequality proven in \cite{limiting_sobolev_inequalities_for_vector_field}:
	
	\begin{theorem}
		\label{gns_for_bva}
		Let $ \ma $ be $ \mr $-elliptic and cancelling. 
		Then there exists some constant
		$ C > 0 $
		such that for all
		$ u \in \testfunctions ( \mr^{ d } , \mr^{ N } ) $,
		we have
		\begin{equation}
			\label{gns_inequality}	
			\norm*{
				u
			}_{ \lspace^{ d / (d - 1 ) } ( \mr^{ d } ) }
			\leq
			C
			\norm*{
				\ma u 
			}_{ \lspace^{ 1 } ( \mr^{ d } ) }.
		\end{equation}
	\end{theorem}
	
	Moreover, it has been shown in \cite{embeddings_for_a_weakly_differentiable_functions_on_domains} that for space dimension $  d \geq  2  $, the assumption of  $ \mc  $-ellipticity always implies that $  \ma $ is cancelling so that we may apply \Cref{gns_for_bva} in our framework. For $ d  =  1 $, we moreover notice that if $ \ma $ is $ \mc  $-elliptic, then inequality \ref{gns_inequality} is just the standard Sobolev inequality since $ \ma = A \partial_{ x } $ for some injective matrix A.
	If we combine this with the extension result \Cref{extension_by_zero_for_bva}, we can thus prove the following.
	
	\begin{theorem}[Sobolev inequality: trace style]
		\label{sobolev_ineq_trace_style}
		Let $ \Omega \subset \mr^{ d }$ be an open and bounded Lipschitz domain and let $ \ma $ be $ \mc $-elliptic.
		Then there exists some constant $ C > 0 $ such that for all 
		$ u \in \bva ( \Omega ) $, we have
		\begin{align*}
			\SwapAboveDisplaySkip
			\norm{
				u
			}_{ \lspace^{ d / ( d - 1 ) } ( \Omega ) }
			& \leq
			C \left(
			\Avariationmeasure{ u } ( \Omega )
			+
			\norm*{
				\Atensor{ \trace{ u } }{ \nu }
			}_{ \lspace^{ 1 } ( \boundary \Omega ) }
			\right)
			\\
			& \leq C
			\left(
			\Avariationmeasure{ u } ( \Omega ) 
			+
			\norm*{
				\trace{ u }
			}_{ \lspace^{ 1 } ( \boundary \Omega ) }
			\right).
		\end{align*}
	\end{theorem}
	
	\begin{proof}
		We extend $ u \in \bva ( \Omega ) $ by zero onto $ \mr^{ d } $.
		By \cref{extension_by_zero_for_bva}, we therefore have $ u \in \bva( \mr^{ d } ) $.
		It follows by the smooth approximation in the strict metric that we find a sequence of test functions
		$ \sequence{ u }{ n } \in \testfunctions ( \mr^{ d } , \mr^{ N } ) $ which converges strictly  to $ \chi_{ \Omega }  u $.
		Moreover we can pass to a non-relabeled subsequence which converges pointwise almost everywhere.
		We therefore have by the Lemma of Fatou and \Cref{gns_for_bva} that
		\begin{align*}
			\norm*{ 
				u
			}_{ \lspace^{ d ( d- 1 ) } ( \Omega ) }
			& \leq
			\liminf_{ n \to \infty }
			\norm*{
				u_{ n }
			}_{ \lspace^{ d / (d - 1) } ( \mr^{ d } ) } 
			\\
			& \leq C
			\liminf_{ n \to \infty }
			\norm*{
				\ma u_{ n }
			}_{ \lspace^{ 1 } ( \mr^{ d } ) }
			\\
			& =
			C \Avariationmeasure{ u } ( \mr^{ d } )
			\\
			& \leq
			C \left(
			\Avariationmeasure{ u } ( \Omega ) 
			+
			\norm*{
				\Atensor{ \trace{ u } }{ \nu }
			}_{ \lspace^{ 1 } ( \boundary \Omega ) }
			\right).
		\end{align*}
		The last inequality follows from \Cref{extension_by_zero_for_bva}.
		This proves the claim.
	\end{proof}
	
	\emergencystretch = 3em
	
	\nocite{continuity_points_via_riesz_potentials_for_c_elliptic_operators}
	\nocite{on_limiting_trace_inequalities_for_vectorial_differential_operators}
	\nocite{formulas_to_represent_functions_by_their_derivatives}
	\nocite{continuity_and_cancelling_operators_of_order_n_on_rn}
	\nocite{a_poincare_type_inequality_for_solutions_of_elliptic_differential_equations}
	\printbibliography
\end{document}